\documentclass[12pt]{amsart}
\usepackage{latexsym}
\usepackage{amsthm}
\usepackage{amsmath}
\usepackage{amsfonts}
\usepackage{amssymb}
\usepackage[dvips]{graphicx}
\usepackage{xypic}
\graphicspath{ {../logic/AEC} }

\input xy
\xyoption{all}

\newtheorem{theo}{Theorem}[section]
\newtheorem{lemma}[theo]{Lemma}
\newtheorem{propo}[theo]{Proposition}
\newtheorem{defi}[theo]{Definition}

\newtheorem{rem}[theo]{Remark}

\newtheorem{problem}[theo]{Problem}

\newcommand\Emb{\operatorname{\bf Emb}}

\newcommand\Set{\operatorname{\bf Set}}

\newcommand\Acc{\operatorname{\bf Acc}}

\newcommand\CAT{\operatorname{\bf CAT}}

\newcommand\Ins{\operatorname{Ins}}
\newcommand\Eq{\operatorname{Eq}}

\newcommand\colim{\operatorname{colim}}

\newcommand\ca{\mathcal {A}}

\newcommand\ch{\mathcal {H}}

\newcommand\ck{\mathcal {K}}
\newcommand\cl{\mathcal {L}}

\date{August 18, 2015}
 
\begin{document}
\title[Limits of abstract elementary classes]
{Limits of abstract elementary classes}
\author[M. Lieberman and J. Rosick\'{y}]
{M. Lieberman and J. Rosick\'{y}}
\thanks{Supported by the Grant agency of the Czech republic under the grant P201/12/G028.} 
\address{
\newline 
Department of Mathematics and Statistics\newline
Masaryk University, Faculty of Sciences\newline
Kotl\'{a}\v{r}sk\'{a} 2, 611 37 Brno, Czech Republic\newline
lieberman@math.mini.cz\newline
rosicky@math.muni.cz
}
 
\begin{abstract}
We show that the category of abstract elementary classes (AECs) and concrete functors is closed under constructions of ``limit type", which generalizes the approach
of Mariano, Zambrano and Villaveces away from the syntactically oriented framework of institutions.  Moreover, we provide a broader view of this closure phenomenon, considering a variety of categories of accessible categories with additional structure, and relaxing the assumption that the morphisms be concrete functors. 
\end{abstract} 
\subjclass{ }

\maketitle
 
\section{Introduction}
One of the main virtues of accessible categories is that they are closed under constructions of limit type (\cite{MP}). 
This should be made precise by considering accessible functors between accessible categories and showing that the resulting
2-category is closed under appropriate limits. These limits can be reduced to products, inserters and equifiers and are called
PIE-limits. Proofs of this result (see \cite{MP}, or \cite{AR}) also show that the category of accessible categories with directed 
colimits and functors preserving directed colimits is closed under PIE-limits. The needed 2-categorical limits are explained both in
\cite{MP} and \cite{AR} and we recommend \cite{Ke} for a more systematic introduction.

Recent papers \cite{BR}, \cite{L} and \cite {LR} have shown that abstract elementary classes (\cite{B}) can be understood as special accessible categories with directed colimits.  In \cite{LR}, in particular, the authors develop a hierarchy of such categories, extending from accessible categories with directed colimits to AECs themselves.  Here we show that each stage in this hierarchy is closed under PIE-limits as well, provided we take the morphisms to be directed colimit preserving functors.  This closure becomes more problematic if we insist that the morphisms be concrete functors: here we see that the iso-fullness axiom for AECs (heretofore unneeded in the category-theoretic analysis thereof) is essential to guarantee the existence of desired limits.

Schematically, our results encompass the categories in the figure below, where the downward-accumulating properties of the objects are described in the left margin, and the properties of the morphisms are listed at the top.  

$$
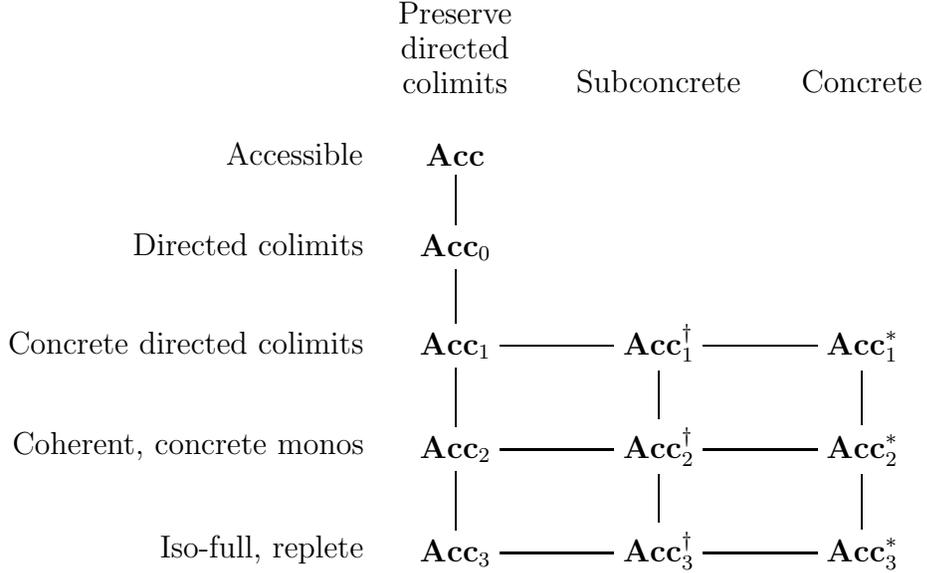
\xymatrix@=1.5pc{
 & & {}\save[]+<0cm,.45cm>*\txt<8pc>{Preserve\\directed\\colimits}\restore & & {}\save[]+<0cm,0cm>*\txt<8pc>{Subconcrete}\restore & & {}\save[]+<0cm,0cm>*\txt<8pc>{Concrete}\restore \\
{}\save[]+<.1cm,.05cm>*\txt<8pc>{Accessible} \restore & & \Acc\ar@{-}[d] & & & & \\
{}\save[]+<-.52cm,.05cm>*\txt<8pc>{Directed colimits}\restore & & \Acc_0\ar@{-}[d] & & & & \\
{}\save[]+<-1.35cm,.05cm>*\txt<16pc>{Concrete directed colimits}\restore & & \Acc_1\ar@{-}[d]\ar@{-}[rr] & & \Acc_1^\dagger\ar@{-}[d]\ar@{-}[rr] & & \Acc_1^\ast\ar@{-}[d] \\
{}\save[]+<-1.32cm,.05cm>*\txt<16pc>{Coherent,~concrete~monos}\restore & & \Acc_2\ar@{-}[d]\ar@{-}[rr] & & \Acc_2^\dagger\ar@{-}[d]\ar@{-}[rr] & & \Acc_2^\ast\ar@{-}[d] \\
{}\save[]+<-.34cm,.05cm>*\txt<8pc>{Iso-full, replete}\restore & & \Acc_3\ar@{-}[rr] & & \Acc_3^\dagger\ar@{-}[rr] & & \Acc_3^\ast}$$

Subconcrete functors, introduced in Definition~\ref{def3.1} below, are a natural generalization of the concrete case. We show that all the pictured categories are closed under PIE-limits in $\Acc$, with the exception of $\Acc_3$, $\Acc_1^\ast$ and $\Acc_3^\ast$. The last category has PIE-limits but it is closed
in $\Acc$ only under inserters and equifiers while products are calculated in $\Acc\downarrow\Set$.
 We note that the objects in categories along the bottom row are (equivalent to) AECs, but equipped with three different notions of morphism, ranging from the most general---functors preserving directed colimits---to a very close generalization of the syntactically-derived functors in \cite{MZV}, namely directed-colimit preserving functors that are concrete, i.e. respect underlying sets.  In particular, the closure result corresponding to the bottom right entry is the promised generalization of \cite{MZV}, shifting it out of the framework of institutions and into a more intrinsic, purely syntax-free characterization.  We consider the precise relationship between our result and that of \cite{MZV} in Remark~\ref{re3.4}.

In fact, our ambitions are broader: inspired by the example of metric AECs, in which directed colimits need not be concrete but $\aleph_1$-directed colimits always are, we consider a second version of this diagram in which we require only that the categories from the third row down have concrete $\kappa$-directed colimits for a given $\kappa$---such categories will be distinguished by the superscript $\kappa$.  In particular, the category $\Acc_3^{\dagger\kappa}$ will consist of $\kappa$-CAECs as defined in \cite{LR1}, with subconcrete functors as morphisms.  We obtain a closure result there as well.

\section{Accessible categories with directed colimits}
Recall that a category $\ck$ is $\lambda$-\textit{accessible}, $\lambda$ a regular cardinal, if it has $\lambda$-directed colimits (i.e. colimits indexed 
by a $\lambda$-directed poset) and contains, up to isomorphism, a set $\ca$ of $\lambda$-presentable objects such that each object of $\ck$ is a 
$\lambda$-directed colimit of objects from $\ca$. Here, an object $K$ is $\lambda$-\textit{presentable} if its hom-functor $\ck(K,-):\ck\to\Set$ preserves 
$\lambda$-directed colimits. A category is \textit{accessible} if it is $\lambda$-accessible for some $\lambda$. A functor $F:\ck\to\cl$ between
$\lambda$-accessible categories is called $\lambda$-accessible if it preserves $\lambda$-directed colimits. $F$ is called \textit{accessible} if it is
$\lambda$-accessible for some $\lambda$. In this way, we get the category $\Acc$ whose objects are accessible categories and morphisms are accessible functors.

\begin{rem}\label{re2.1}
{
\em
We work in the G\" odel-Bernays set theory. Thus a category $\ck$ is a class of objects together with a class $\ck(A,B)$ of morphisms
$A\to B$ for each object $A$ and $B$. It is called \textit{locally small} if all $\ck(A,B)$ are sets. Any accessible category is locally
small. It is important to observe that $\Acc$ is a category which is not locally small. The reason is that a $\lambda$-accessible functor
$F:\ck\to\cl$ is determined by its restriction on the full subcategory $\ca$ of $\lambda$-presentable objects.
}
\end{rem}

We may regard $\Acc$ as a 2-category where the 2-cells are natural transformations. As noted above, $\Acc$ is closed under appropriate 2-limits, namely PIE-limits, where ``PIE''  abbreviates ``products,'' ``inserters'' and ``equifiers.'' This means that these 2-limits exist in $\Acc$ and are calculated 
in the non-legitimate category $\CAT$ of categories, functors and natural transformations. It follows that $\Acc$ is closed under lax limits
and under pseudolimits (see \cite{MP} or \cite{AR}). 

Recall that, given functors $F,G:\ck\to\cl$, the \textit{inserter category} $\Ins(F,G)$ is the subcategory of the comma category $F\downarrow G$ 
consisting of all objects $f:FK\to GK$ and all morphisms
$$
\xymatrix@=4pc{
FK \ar [r]^{f} \ar [d]_{Fk}& GK \ar [d]^{Gk}\\
FK'\ar [r]_{f'}& GK'
}
$$
The \textit{projection functor} $P:\Ins(F,G)\to\ck$ sends $f:FK\to GK$ to $K$. The universal property of $\Ins(F,G)$ is the existence
of a natural transformation $\varphi:FP\to GP$ (given as $\varphi_{Pf}=f$) in the sense that for any $H:\ch\to\ck$ with $\psi:FH\to GH$ there
is a unique $\bar{H}:\ch\to\Ins(F,G)$ such that $P\bar{H}=H$ and $\varphi H=\psi$ (see \cite{Ke}). Since $\Acc$ is full in $\CAT$ with respect
to 2-cells, we can ignore the 2-dimensional aspect of universality. 

Given functors $F,G:\ck\to\cl$ and natural transformations $\varphi,\psi:F\to G$, the \textit{equifier} $\Eq(\varphi,\psi)$
is the full subcategory of $\ck$ consisting of all objects $K$ such that $\varphi_K=\psi_K$. Let $P:\Eq(\varphi,\psi)\to\ck$ be the inclusion.
The universal property of $\Eq(\varphi,\psi)$ is that $\varphi P=\psi P$ and for any $H:\ch\to\ck$ with $\varphi H=\psi H$ there is a unique
$\bar{H}:\ch\to\ck$ such that $P\bar{H}=H$ (see \cite{Ke}); the 2-dimensional aspect of universality can be ignored again.

We now consider accessible categories having all directed colimits. Let $\Acc_0$ be the 2-category whose objects are accessible categories with directed colimits, morphisms are functors preserving directed colimits and 2-cells are natural transformations. 

\begin{theo}\label{th2.2}
$\Acc_0$ is closed under PIE-limits in $\Acc$.
\end{theo}
\begin{proof}
Let $\ck_i$, $i\in I$ be accessible categories with directed colimits. Following \cite{AR} 2.67, the product $\prod_{i\in I}\ck_i$
is an accessible category.  Clearly, it has all directed colimits and the projections $P_i:\prod\ck_i\to\ck_i$ preserve them. Let $\cl$
be an accessible category with directed colimits and $Q_i:\cl\to\ck_i$ functors preserving directed colimits. Then the induced functor
$\cl\to\prod\ck_i$ preserves directed colimits. Hence $\prod\ck_i$ is the product in $\Acc_0$.

Let $\ck,\cl$ be accessible categories with directed colimits and $F,G:\ck\to\cl$ be functors preserving directed colimits. Following
\cite{AR} 2.72, $\Ins(F,G)$ is an accessible category which clearly has directed colimits. Let $\ch$ be an accessible category with directed colimits,
$H:\ch\to\ck$ preserve directed colimits and and $\psi:FH\to GH$ a natural transformation. Then the induced functor $\bar{H}:\ch\to\Ins(F,G)$
preserves directed colimits. Hence $\Ins(F,G)$ is an inserter in $\Acc_0$.

Finally, let $\ck,\cl$ be accessible categories with directed colimits, $F,G:\ck\to\cl$ functors preserving directed colimits 
and $\varphi,\psi:F\to G$ natural transformations. Following \cite{AR} 2.76, $\Eq(\varphi,\psi)$ is an accessible category. Again, it is clear that $\Eq(\varphi,\psi)$ has all directed colimits.  Let $\ch$ be an accessible category with directed colimits and $H:\ch\to\ck$ a functor preserving directed colimits
with $\varphi H=\psi H$. Then the induced functor $\bar{H}:\ch\to\Eq(\varphi,\psi)$ preserves directed colimits. Hence $\Eq(\varphi,\psi)$
is an equifier in $\Acc_0$.
\end{proof}

We say that $(\ck,U)$ is an \textit{accessible category with concrete directed colimits} if $\ck$ is an accessible category with directed colimits 
and $U:\ck\to\Set$ is a faithful functor to the category of sets that preserves directed colimits. Let $\Acc_1$ be the full sub-2-category of $\Acc_0$
consisting of accessible categories with concrete directed colimits. In particular, morphisms in $\Acc_1$ are functors preserving directed colimits.

\begin{theo}\label{th2.3}
$\Acc_1$ is closed under PIE-limits in $\Acc$.
\end{theo}
\begin{proof}
We must show that PIE-limits of accessible categories with concrete directed colimits have concrete directed colimits. This is evident
for inserters and equifiers because, in the first case, the projection functor $P:\Ins(F,G)\to\ck$ is faithful and, in the second case,
$\Eq(\varphi,\psi)$ is a full subcategory of $\ck$. Consider accessible categories with concrete directed colimits $(\ck_i,U_i)$, $i\in I$.
Then the functor  $U:\prod_{i\in I}\ck_i\to\Set$ sending $(A_i)_{i\in I}$ to $\coprod_{i\in I}U_iA_i$ is faithful. Since
$$
\colim\coprod_{i\in I} U_iA_i\cong\coprod_{i\in I}\colim U_iA_i,
$$
$\prod_{i\in I}\ck_i$ is an accessible category with concrete directed colimits.  
\end{proof}

\begin{rem}\label{re2.4}
{
\em
(1) We could also consider the subcategory $\Acc_1^\ast$ having the same objects as $\Acc_1$ but whose morphisms are \textit{concrete} functors 
$F:\ck_1\to\ck_2$ preserving directed colimits.  By ``concrete,'' we mean that $F$ commutes with the relevant underlying set functors, i.e. $U_2F=U_1$. 
The category $\Acc_1^\ast$ is closed in $\Acc$ under inserters and equifiers but not under products. In fact, we are in the comma category $\Acc\downarrow\Set$ where $\prod_{i\in I}(\ck_i,U_i)$ is the multiple pullbacks of $U_i$ over $\Set$. While $\Acc$ has multiple {\it pseudopullbacks}, it does not have multiple pullbacks. For multiple pullbacks, we would need all of the functors  $U_i$ to be \textit{transportable} in the sense that for any isomorphism $f:U_iA\to X$ there is a unique isomorphism $\overline{f}:A\to B$ such that $U_i(\overline{f})=f$ (this also implies $U_i\overline{B}=X$). Then a multiple pullback of $U_i$
is equivalent to their multiple pseudopullback and thus it belongs to $\Acc$. This is done for a pullback in \cite{MP} 5.1.1 and the multiple
case is analogous.

(2) Theorem \ref{th2.3} is also valid for the full sub-2-category $\Acc^\kappa_1$ of $\Acc_0$ consisting of accessible categories with directed
colimits where $\kappa$-directed colimits are concrete. These categories appear in \cite{LR1} in connection with metric abstract elementary classes.
}
\end{rem}

An accessible category $(\ck,U)$ with concrete directed colimits is \textit{coherent} if for each commutative triangle 
$$
\xymatrix@=3pc{
UA \ar[rr]^{U(h)}
\ar[dr]_{f} && UC\\
& UB \ar[ur]_{U(g)}
}
$$
there is $\overline{f}:A\to B$ in $\ck$ such that $U(\overline{f})=f$.

We say that morphisms of $\ck$ are \textit{concrete monomorphisms} if any morphism of $\ck$ is a monomorphism which is preserved by $U$. 
Let $\Acc_2$ be the full sub-2-category of $\Acc_1$ consisting of coherent accessible categories with concrete monomorphisms.

\begin{theo}\label{th2.5}
$\Acc_2$ is closed under PIE-limits in $\Acc$.
\end{theo}
\begin{proof}
Since there is no problem with concrete monomorphisms, we have to show that PIE-limits of coherent accessible categories are coherent. This is evident 
for equifiers because $\Eq(\varphi,\psi)$ is a full subcategory of $\ck$. Consider coherent accessible categories $(\ck_i,U_i)$, $i\in I$. We have to show 
that $U:\prod_{i\in I}\ck_i\to\Set$ sending $(A_i)_{i\in I}$ to $\coprod_{i\in I}U_iA_i$ is coherent. 
Consider a commutative triangle
$$
\xymatrix@=3pc{
U(A_i) \ar[rr]^{U(h)}
\ar[dr]_{f} && U(C_i)\\
& U(B_i) \ar[ur]_{U(g)}
}
$$
and $a\in U_iA_i$. Assume that $fa_i\in U_jB_j$ for $j\neq i$. Then $(Ug)fa_i\in U_jC_j$ and $(Uh)a_i\in U_iC_i$, which is impossible.
Thus $f=\coprod_{i\in I}f_i$. Since each $U_i$ is coherent, there are morphisms $\overline{f}_i:A_i\to B_i$ such that $U(\overline{f}_i)=f$.
Hence $\prod_{i\in I}\ck_i$ is coherent.  

Consider morphisms $F,G:\ck\to\cl$ in $\Acc_2$. We have to show that the composition
$$
\Ins(F,G) \xrightarrow{\quad  P\quad} \ck
             \xrightarrow{\quad U\quad} \Set
$$
is coherent. Consider a commutative triangle
$$
\xymatrix@=3pc{
UPf_1 \ar[rr]^{UP(h)}
\ar[dr]_{f} && UPf_3\\
& UPf_2 \ar[ur]_{UP(g)}
}
$$
where $f_i:FK_i\to GK_i$, $i=1,2,3$. Thus we have a commutative triangle
$$
\xymatrix@=3pc{
UK_1 \ar[rr]^{U(h)}
\ar[dr]_{f} && UK_3\\
& UK_2 \ar[ur]_{U(g)}
}
$$
and, since $U$ is coherent, we have $f=U\overline{f}$.
Thus we get the diagram
$$
\xymatrix@=3pc{
FK_1 \ar[r]^{f_1} \ar[d]_{F\overline{f}} & GK_1\ar[d]^{G\overline{f}}\\
FK_2 \ar[r]^{f_2} \ar[d]_{Fg} & GK_2
\ar[d]^{Gg}\\
FK_3\ar[r]_{f_3} &GK_3
}
$$
where the outer rectangle and the bottom square commute. Since $Gg$ is a monomorphism, the upper square commutes as well.
Hence $\overline{f}:f_1\to f_2$ is a morphism in $\Ins(F,G)$ and $f=UP\overline{f}$. Therefore  $PU$ is coherent.
\end{proof}

\begin{rem}\label{re2.6}
{
\em
(1) The assumption that objects of $\Acc_2$ have concrete monomorphisms was needed in the proof of closure under inserters.

(2) Theorem \ref{th2.5} is also valid for the full sub-2-category $\Acc^\kappa_2$ of $\Acc^\kappa_1$ consisting of coherent accessible
categories with directed colimits and concrete monomorphisms. 
}
\end{rem}

Abstract elementary classes can be characterized as coherent accessible categories $\ck$ with directed colimits and with concrete monomorphisms satisfying two additional conditions dealing with finitary function and relation symbols interpretable in $\ck$ (see \cite{LR}). Here, \textit{finitary relation symbols interpretable in} $\ck$ are subfunctors $R$ of $U^n=\Set(n,U-)$ where $n$ is a finite cardinal. \textit{Finitary function symbols interpretable in} $\ck$ are natural transformations $h:U^n\to U$. Since $n$-ary function symbols can be replaced by $(n+1)$-ary relation symbols, we can confine ourselves to finitary relation symbols interpretable in $\ck$. Let $\Sigma_\ck$ consists of those finitary relation symbols $R$ interpretable in $\ck$ for which $\ck$-morphisms 
$f:A\to B$ behave as embeddings. This means that if $(Uf)^n(a)\in R_B$ then $a\in R_A$. We get the functor $E:\ck\to\Emb(\Sigma_\ck)$ where $\Emb(\Sigma_\ck)$
is the category of $\Sigma_\ck$-structures whose morphisms are substructure embeddings. Now, $\ck$ is
an abstract elementary class if and only if the functor $E$ is \textit{full with respect to isomorphisms} and \textit{replete}. The first condition means that 
if $f:EA\to EB$ is an isomorphism then there is an isomorphism $\overline{f}:A\to B$ with $E\overline{f}=f$. We also say that $R\in\Sigma_\ck$ \textit{detect
isomorphisms}; this condition makes $\ck$ equivalent to an abstract elementary class. The second condition means that if $EA$ is isomorphic to $X$ then there 
is $B\in\ck$ such that $A$ is isomorphic to $B$ and $EB=X$.  

We note that abstract elementary classes are commonly presented via an embedding $\ck\to\Emb(\Sigma)$. In this case, $\Sigma\subseteq\Sigma_\ck$ and, in fact, $\Sigma_\ck$ is
the largest relational signature in which $\ck$ can be presented.

Let $\Acc_3$ be the full sub-2-category of $\Acc_2$ consisting of categories equivalent to abstract elementary classes.

\begin{propo}\label{prop2.7} $\Acc_3$ is closed under products and equifiers in $\Acc$.
\end{propo}
\begin{proof}
The closedness under equifiers immediately follows from the fact that $\Eq(\varphi,\psi)\to\ck$ is a replete, full embedding. Consider $(\ck_i,U_i)$,
$i\in I$, in $\Acc_3$. Given $n$-ary relation symbols $R_i\in\Sigma_{\ck_i}$ where $i\in I$, we get the $n$-ary relation symbol $R=\coprod_iR_i$
belonging to $\Sigma_{\prod_i\ck_i}$. It includes unary interpretable relation symbols given by $R_j=U_j$ and $R_i=\emptyset$ for $i\neq j$. It is easy to see
that these $R$ detect isomorphisms. Thus $E: \prod_i \ck_i\to \Emb(\Sigma_{\prod \ck_i})$ is full with respect to isomorphisms. Clearly, it is replete. 
\end{proof}

\begin{rem}\label{re2.8}
{
\em
In the case of inserters, any finitary relation symbol $R$ interpretable in $\ck$ yields the finitary relation symbol $RP$ interpretable
in $\Ins(F,G)$. Let $f_i:FK_i\to GK_i$ for $i=1,2$ and 
$$
f:UK_1=UPf_1\to UPf_2=UK_2
$$ 
be a bijection such that $f^n$ induces a bijection between $Sf_1$ and $Sf_2$ for each $n$-ary relation symbol $S$ interpretable in $\Ins(F,G)$. 
By taking $S=RP$ we get that $f^n$ induces a bijection between $RK_1$ and $RK_2$ for each $n$-ary relation symbol $R$ interpretable in $\ck$. 
Since $E:\ck\to\Emb(\Sigma_\ck)$ is full with respect to isomorphisms, there is an isomorphism $\overline{f}:K_1\to K_2$ with $U\overline{f}=f$. 
But we do not know whether $\overline{f}:f_1\to f_2$ is a morphism, i.e., whether the square
$$
\xymatrix@=4pc{
FK_1 \ar [r]^{f_1} \ar [d]_{F\overline{f}}& GK_1 \ar [d]^{G\overline{f}}\\
FK_2\ar [r]_{f_2}& GK_2
}
$$
commutes.  
}
\end{rem}

\begin{problem}\label{pr2.9}
{
\em
Is $\Acc_3$ closed under inserters in $\Acc$?
}
\end{problem}

\section{Abstract elementary classes}
\begin{defi}\label{def3.1}
{
\em
Let $(\ck_1,U_1)$ and $(\ck_2,U_2)$ be concrete categories. We say that a functor $H:\ck_1\to\ck_2$ is \textit{subconcrete}
if there is a natural monotransformation $\alpha:U_2H\to U_1$ such that if $(U_1f)a\in U_2HB$ then $a\in U_2HA$ for each $f:A\to B$
in $\ck_1$.
}
\end{defi}

This means that $U_2H$ is a unary relation symbol belonging to $\Sigma_{\ck_1}$. Any concrete functor is subconcrete.
Since a composition of subconcrete functors is subconcrete, we get the subcategory $\Acc_1^\dagger$ of $\Acc_1$ consisting of accessible categories
with concrete directed colimits and subconcrete functors preserving directed colimits. Analogously, we get the full subcategory $\Acc_2^\dagger$ 
of $\Acc_1^\dagger$ consisting of coherent accessible categories and concrete monomorphisms whose morphisms are subconcrete functors preserving directed colimits. Finally, we have the category
$$
\Acc_3^\dagger=\Acc_3\cap\Acc_2^\dagger
$$
of categories equivalent to abstract elementary classes and subconcrete functors preserving directed colimits.

\begin{theo}\label{th3.2} $\Acc_1^\dagger$, $\Acc_2^\dagger$ and $\Acc_3^\dagger$ are closed under PIE-limits in $\Acc$.
\end{theo}
\begin{proof}
In $\Acc_1^\dagger$ and $\Acc_2^\dagger$ the case of equifiers and inserters is evident because $P:\Eq(\varphi,\psi)\to\ck$ and $P:\Ins (F,G)\to\ck$ are concrete. In the case of products, the projections $P_i:\prod_i\ck_i\to\ck_i$ are subconcrete -- take the coproduct injections $U_iP_i\to U$ where
$U:\prod\ck_i\to\Set$ is from the proof of \ref{th2.3}. 

We have to prove that $\Acc_3^\dagger$ is closed under inserters. Let $(\ck_1,U_1)$ and $(\ck_2,U_2)$ be abstract elementary classes
and $F,G:\ck_1\to\ck_2$ subconcrete functors. First, in the notation of \ref{re2.8}, we show that the square
$$
\xymatrix@=4pc{
FK_1 \ar [r]^{f_1} \ar [d]_{F\overline{f}}& GK_1 \ar [d]^{G\overline{f}}\\
FK_2\ar [r]_{f_2}& GK_2
}
$$
commutes. Since $F$ and $G$ are subconcrete, we get unary relation symbols $U_2F,U_2G\in\Sigma_\ck$. Hence we have unary relation symbols
$$
U_2FP,U_2GP\in \Sigma_{\Ins(F,G)}.
$$ 
Thus we have a binary relation symbol $R\in \Sigma_{\Ins(F,G)}$ such that $(a,b)\in R_g$, $g:FK\to GK$, if $a\in U_2FPg$,
$b\in U_2GPg$ and $b=(U_2g)a$. To see that the above square commutes, notice that $(a,(U_2f_1)a)\in R_{f_1}$ for each $a\in U_2FK_1$.
It follows that $((U_2F\overline{f})a,U_2(G(\overline{f})f_1)a)\in R_{f_2}$, and therefore that $U_2(G(\overline{f})f_1)a=U_2(f_2F\overline{f})a$ for each $a\in U_2FK_1$. 
Hence 
$$
U_2(G(\overline{f})f_1)=U_2(f_2F\overline{f})
$$
and, since $U_2$ is faithful,
$$
G(\overline{f})f_1=f_2F\overline{f}.
$$

It remains to show that $(\Ins(F,G),U_1P)$ is replete. But, having $f:FA\to GA$ in $\Ins(F,G)$ and an isomorphism $h:A\to B$, then 
$g=G(h)fF(h)^{-1}:FB\to GB$ belongs to $\Ins(F,G)$ and $h:f\to g$ is an isomorphism.
\end{proof}

Finaly, we have the category
$$
\Acc_3^\ast=\Acc_3\cap\Acc_1^\ast
$$
of categories equivalent to abstract elementary classes whose morphisms are concrete functors preserving directed colimits.

\begin{theo}\label{th3.3}
$\Acc_3^\ast$ has PIE-limits.
\end{theo}
\begin{proof}
Since concrete functors are subconcrete, $\Acc_3^\ast$ is closed in $\Acc$ under inserters and equifiers. Products $\prod_{i\in I}(\ck_i,U_i)$
are calculated in $\Acc\downarrow\Set$, i.e., they are multiple pullbacks. Since any abstract elementary class $(\ck,U)$ has $U$ transportable,
multiple pullbacks are equivalent to multiple pseudopullbacks (see \ref{re2.4}). Following \ref{th3.2}, $\Acc_3^\dagger$ is closed
in $\Acc$ under PIE-limits and, consequently, under pseudolimits. Thus $\prod_{i\in i}(\ck_i,U_i)$ belongs to $\Acc_3^\ast$ and is the product
of $(\ck_i,U_i)$ there.  
\end{proof}
 
\begin{rem}\label{re3.4}
{
\em
(1) Let $H:\ck_1\to\ck_2$ be a morphism in $\Acc_3^\dagger$. Since $(U_2H)^n$ is an $n$-ary relation symbol belonging to $\Sigma_{\ck_1}$, we
get an embedding of signatures $\overline{H}:\Sigma_{\ck_2}\to\Sigma_{\ck_1}$ sending $R$ to $RH$. In particular, it sends $U_2$ to $U_2H$.
This induces the subconcrete functor 
$\Emb(\overline{H}):\Emb(\Sigma_{\ck_1})\to\Emb(\Sigma_{\ck_2})$ given by taking reducts. The square
$$
\xymatrix@=4pc{
\ck_1 \ar [r]^{H} \ar [d]_{E_1}& \ck_2 \ar [d]^{E_2}\\
\Emb(\Sigma_{\ck_1})\ar [r]_{\Emb(\overline{H})}& \Emb(\Sigma_{\ck_2})
}
$$
clearly commutes. 

If $H$ is concrete then $\Emb(\overline{H})$ is concrete as well. This relates our morphisms of abstract elementary classes to the syntactically-derived morphisms considered in \cite{MZV}.

(2) On the other hand, let $G:\Sigma_2\to\Sigma_1$ be an embedding of signatures. Let $\ck_1\to\Emb(\Sigma_1)$ and $\ck_2\to\Emb(\Sigma_2)$ be abstract elementary classes---in the classical sense---presented in signatures $\Sigma_1$ and $\Sigma_2$: note that, when paired with their natural underlying set functors $U_i:\ck_i\to\Set$, they satisfy the purely category-theoretic characterization of AECs following Remark 2.6.  Moreover, let $H:\ck_1\to\ck_2$ be a functor such that the square
$$
\xymatrix@=4pc{
\ck_1 \ar [r]^{H} \ar [d]_{}& \ck_2 \ar [d]^{}\\
\Emb(\Sigma_1)\ar [r]_{\Emb(G)}& \Emb(\Sigma_2)
}
$$
commutes. Since $\Emb(G)$ is concrete, $H$ is a morphism in $\Acc_3^\ast$. These are precisely the morphisms of abstract
elementary classes considered in \cite{MZV}. 

More generally, consider relational signatures $\Sigma_1,\Sigma_2$ and let $L(\Sigma_1),L(\Sigma_2)$ be the corresponding languages, i.e., sets of all formulas 
of $\Sigma_1,\Sigma_2$. Consider a mapping $-^\ast:\Sigma_2\to\Sigma_1$ of signatures preserving the arity of symbols and let $P$ be a unary relation symbol 
in $\Sigma_1$. Let $G:L(\Sigma_2)\to\L(\Sigma_1)$ be a morphism of languages sending each ($n$-ary) relation symbol $R$ in $\Sigma_1$ to $P^n\wedge R^\ast$. This defines $G$ on the atomic formulas of $L(\Sigma_1)$; we extend it recursively to all of $L(\Sigma_1)$.  In particular, $G$ sends $=$ to the equality $=_P$ on $P$. Then $\Emb(G):\Emb(\Sigma_1)\to\Emb(\Sigma_2)$ 
is a subconcrete functor. Let $\ck_1\to\Emb(\Sigma_1)$ and $\ck_2\to\Emb(\Sigma_2)$ be abstract elementary classes and $H:\ck_1\to\ck_2$ be a functor 
such that the square
$$
\xymatrix@=4pc{
\ck_1 \ar [r]^{H} \ar [d]_{}& \ck_2 \ar [d]^{}\\
\Emb(\Sigma_1)\ar [r]_{\Emb(G)}& \Emb(\Sigma_2)
}
$$
commutes. Then $H$ is a morphism in $\Acc_3^\dagger$.  

(3) Let $\ck$ be the category of infinite sets and monomorphisms. Then $\ck$ is an abstract elementary class in the empty signature $\Sigma_2$.
Let $\Sigma_1$ contain just $P$ and $=_P$ from (2) and $G:\Sigma_2\to\Sigma_1$ be the corresponding morphism of languages. We also have $H:\Sigma_1\to\Sigma_2$
sending $=_P$ to $=$ and $P$ to the formula $x=x$. In this way, we make $\ck$ isomorphic to an abstract elementary class in the signature $\Sigma_1$,
where we interpret objects of $\ck$ as $\Sigma_1$-structures $K$ with $P_K$ infinite and $(\neg P)_K$ countable (see \cite{BR}, 5.8(3) motivated
by \cite{K} 2.10).

(4) Theorem \ref{th3.2} is also valid for categories $\Acc_1^{\dagger\kappa}$, $\Acc_2^{\dagger\kappa}$ and $\Acc_3^{\dagger\kappa}$ where
$\Acc_1$ is replaced by $\Acc^\kappa_1$.

Analogously, Theorem \ref{th3.3} is valid for $\Acc_3^{\ast\kappa}$.
}
\end{rem}

\begin{lemma}\label{le3.5} Any morphism in $\Acc_3^\dagger$ is coherent and transportable.
\end{lemma}
\begin{proof}
Consider the square from \ref{re3.4}(1).
Since the functor $\Emb(\overline{H})$ is coherent, the composition $\Emb(\overline{H})E_1$ is coherent as well. 
Since $E_2$ is faithful, $H$ is coherent. 

Since $\Sigma_{\ck_1}$ and $\Sigma_{\ck_2}$ contain only relation symbols, the functor $\Emb(\overline{H})$ is surjective on objects and full
(by interpreting the missing relations as empty). Consider an isomorphism $f:HA\to B$. We get the isomorphism
$$
E_2f:\Emb(\overline{H})E_1A=E_2HA\to E_2B=\Emb(\overline{H})\tilde{B}. 
$$
and thus the isomorphism $\tilde{f}:E_1A\to\tilde{B}$ such that $\Emb(\overline{H})\tilde{f}=E_2f$. Since $E_1$ is transportable, there is an isomorphism 
$\overline{f}:A\to\overline{B}$ such that $E_1\overline{B}=\tilde{B}$ and $E_1\overline{f}=\tilde{f}$. Clearly, $H\overline{f}=f$. Thus $H$ is transportable.
\end{proof}

\begin{rem}
{
\em
Following \ref{le3.5} and \ref{re2.4}, pullbacks in $\Acc_3^\dagger$ are equivalent to pseudopullbacks. Thus \ref{th3.2} implies that $\Acc_3^\dagger$
is closed in $\CAT$ under pullback. Consequently, the same holds for $\Acc_3^\ast$, which was proved in \cite{MZV}.
}
\end{rem}
%In the case of $\Acc_3^{\dagger\kappa}$, this guarantees that the 2-category of $\kappa$-AECs (see \cite{LR1}) with %subconcrete functors as morphisms is closed under PIE-limits in $\Acc$.  That this result does not transfer comfortably to% $\Acc_3^{\ast\kappa}$, $\kappa$-AECs with concrete functors as morphisms, should come as no surprise: in mAECs, the% motivating example, the primary focus is on dense subsets of the underlying sets of models, whereas the underlying sets themselves---what we would aim to preserve with concrete functors---are essentially incidental. 

\end{document}